\documentclass[a4paper,pdftex,reqno,10pt]{amsart}

\usepackage[english]{babel} 

\usepackage{amsmath,amsfonts}
\usepackage[autostyle]{csquotes}
\usepackage{hyperref}

\DeclareMathOperator{\PGammaL}{P\Gamma{}L}
\DeclareMathOperator{\st}{st}
\DeclareMathOperator{\GL}{GL}

\DeclareMathOperator{\Aut}{Aut}
\DeclareMathOperator{\im}{im}

\newcommand{\gaussmnum}[3]{\left[\begin{smallmatrix}{#1}\\{#2}\end{smallmatrix}\right]_{#3}}
\newcommand{\gaussmset}[2]{\left[\begin{smallmatrix}{#1}\\{#2}\end{smallmatrix}\right]}
\newcommand{\F}{\ensuremath{\mathbb{F}}}
\newcommand{\Fq}{\ensuremath{\F_q}}
\newcommand{\Fqn}{\ensuremath{V}}
\newcommand{\PFqn}{\ensuremath{L(\Fqn)}}
\newcommand{\Gqnk}{\ensuremath{\gaussmset{\Fqn}{k}}}
\newcommand{\I}[2]{\ensuremath{\mathcal{I}\left(#1,#2\right)}}
\newcommand{\dham}{\mathrm{d}_{\mathrm{Ham}}}

\newcommand{\sdist}{\mathrm{d}_{\mathrm{s}}}

\newtheorem{proposition}{Proposition}
\newtheorem{corollary}{Corollary}

\newtheorem{theorem}{Theorem}
\newtheorem{lemma}{Lemma}

\begin{document}

\title{A new upper bound for subspace codes}

\author[Heinlein]{Daniel Heinlein}
\address{Daniel Heinlein, University of Bayreuth, 95440 Bayreuth, Germany, daniel.heinlein@uni-bayreuth.de}

\author[Kurz]{Sascha Kurz}
\thanks{The work was supported by the ICT COST Action IC1104 
and grants KU 2430/3-1, WA 1666/9-1 -- ``Integer Linear Programming Models for Subspace Codes and 
Finite Geometry'' -- from the German Research Foundation.}
\address{Sascha Kurz, University of Bayreuth, 95440 Bayreuth, Germany, sascha.kurz@uni-bayreuth.de}


\maketitle

\begin{abstract}It is shown that the maximum size $A_2(8,6;4)$ of a binary subspace code of packet length $v=8$, minimum subspace distance
$d=4$, and constant dimension $k=4$ is at most $272$. In Finite Geometry terms, the maximum number of solids in 
$\operatorname{PG}(7,2)$, mutually intersecting in at most a point, is at most $272$. Previously, the best known upper bound 
$A_2(8,6;4)\le 289$ was implied by the Johnson bound and the maximum size $A_2(7,6;3)=17$ of partial plane spreads in 
$\operatorname{PG}(6,2)$.  The result was obtained by combining the classification of subspace codes with 
parameters $(7,17,6;3)_2$ and  $(7,34,5;\{3,4\})_2$ with integer linear programming techniques. The classification of 
$(7,33,5;\{3,4\})_2$ subspace codes is obtained as a byproduct.  
\end{abstract}

\section{Introduction}

For a prime power $q > 1$ let {\Fq} be the field with $q$ elements and $V\cong{\F}_q^v$ a $v$-dimensional vector space over {\Fq}. 
The set ${\PFqn}$ of all subspaces of $V$, or flats of the projective geometry 
$\operatorname{PG}(V)\cong \operatorname{PG}(\Fq)=:\operatorname{PG}(v-1,q)$, forms 
a metric space with respect to the subspace distance defined by $\sdist(U,W) = \dim(U+W)-\dim(U \cap W)=\dim(U)+\dim(W)-2\dim(U \cap W)$. 
The metric space $({\PFqn},\sdist)$ may be viewed as a $q$-analogue of the Hamming space $(\F_2^v,\dham)$ used in conventional coding
theory via the subset-subspace analogy~\cite{knuth71a}. In their seminal paper \cite{MR2451015} K{\"o}tter and Kschischang motivate 
coding on {\PFqn} via error correcting random network coding, see  \cite{MR1768542}. 
By {\Gqnk} we denote the set of all $k$-dimensional subspaces in {\Fqn}, where $0\le k\le v$,  
which has size $\gaussmnum{v}{k}{q}:= \prod_{i=1}^{k} \frac{q^{v-k+i}-1}{q^i-1}$.
A \emph{subspace code} is a subset of {\PFqn} and each element is called \emph{codeword}.   
By $(v,N,d;K)_q$ we denote a subspace code in {\Fqn} with minimum (subspace) distance $d$ and size  
$N$, where the dimensions of each codeword is contained in $K \subseteq \{0,1,\ldots,v\}$. As usual, a subspace code 
$C$ has the \emph{minimum distance} $d$, if $d \le \sdist(U,W)$ for all $U \ne W \in C$ and equality is attained at least once. 
The corresponding maximum size is denoted by $A_q(v,d;K)$. Its determination is called \emph{Main Problem of Subspace Coding} at several 
places. The \emph{dimension distribution} of a subspace code $C$ in {\Fqn} is a string $0^{m_0} 1^{m_1} \ldots v^{m_v}$ such that the 
number of $i$-dimensional codewords in $C$ is $m_i$, where entries with $m_i=0$ are commonly omitted. In the special case 
where the set $K$ of codeword dimensions is a singleton we speak of a \emph{constant dimension code} (cdc) 
and abbreviate $K=\{k\}$ by just $k$ in the above notation.   
 
In a $(v,N,d;k)_q$ code the minimum distance $d$ has to be an even number satisfying $2\le d\le 2k$. If $d=2k$ one speaks of partial 
$k$-spreads. While there is a lot of recent research on $A_2(v,2k;k)$, i.e., partial spreads, see e.g.\ 
\cite{honold2016partial,kurz2016improved,kurz2017,nastase2016maximumII,nastase2016maximum}, the known upper bounds for $A_2(v,d;k)$ 
with $d<2k$ are relatively straightforward. Besides recursive implications of the Johnson bound 
\begin{equation}
  \label{eq_johnson}
  A_q(v,d;k) \le \left\lfloor\frac{q^v-1}{q^k-1} \cdot A_q(v-1,d;k-1)\right\rfloor,
\end{equation}
see \cite[Theorem~4]{MR2810308}, the only improvement
$A_2(6,4;3)=77<81$ (for $q\le 9$ and $v\le 19$) was obtained in \cite{MR3329980}. In this paper we add 
$A_2(8,6;4)\le 272<289$ to this very short list. Assuming $4\le d\le 2k-2$, the only known case where the Johnson bound
is attained is given by $A_2(13,4;3)=1597245$ \cite{Braun16}. For numerical values of the known lower and upper bounds 
the sizes of subspace codes we refer the reader to the online tables \url{http://subspacecodes.uni-bayreuth.de} 
associated with \cite{HKKW2016Tables}. A survey on Galois geometries and coding theory can be found in \cite{Etzion2016}. 

The so-called Echelon--Ferrers construction, see e.g.~\cite{MR2589964}, gives $A_2(8,6;4)\ge 257$. More precisely, the corresponding 
code is a lifted maximum rank distance (MRD) code plus a codeword. Codes containing the lifted MRD code have a size of at most 
$257$, see  \cite[Theorem~10]{MR3015712}.  

The remaining part of the paper is structured as follows. In Section~\ref{sec_preliminaries} we collect theoretical preliminaries that are 
used later on to deduce the presence of certain substructures of a constant dimension code of relatively large size. Our main result, the 
proposed upper bound $A_2(8,6;4)\le 272$, is concluded in Section~\ref{main_sec} based on integer linear programming techniques. Here, the 
mentioned substructures are prescribed using classification results, see the webpage associated with \cite{HKKW2016Tables}, where 
the corresponding lists can be downloaded.

In Section~\ref{sec_alternative_approach} we present alternative approaches leading to the same result, i.e., independently verifying it. As 
a byproduct we classify the $(7,33,5;\{3,4\})_2$ subspace codes up to isomorphism in Theorem~\ref{thm_class_3}. We close with a summary and 
a discussion of possible future research in Section~\ref{sec_conclusion}. 

\section{Preliminaries}
\label{sec_preliminaries}
Later on we will classify special classes of subspace codes up to isomorphism. To this end, we remark that for $v\ge 3$ the automorphism group of 
the metric space $({\PFqn},\sdist)$ is given by the group $\langle \PGammaL(\Fqn), \pi \rangle$, with $\pi:\gaussmset{\Fqn}{k} \mapsto 
\gaussmset{\Fqn}{v-k}, U \mapsto U^{\perp}$ (fixing an arbitrary non-degenerated bilinear form for $^{\perp}$).
In particular for a subspace code $C$ with parameters $(n,N,d;K)_q$ the code $C^\perp = \pi(C) = \{U^\perp \mid U \in C\}$ is called the 
\emph{orthogonal code} of $C$ and it has the same parameters $(v,N,d;v-K)_q$, i.e., $A_q(v,d;K)=A_q(v,d;v-K)$, where $v-K=\{v-k\mid k\in K\}$.

In order to describe some structural properties of a constant dimension code and to give bounds we will consider incidences with fixed subspaces. To this end, 
let $\I{S}{X}$ be the set of subspaces in $S \subseteq \PFqn$ that are incident to $X \le \Fqn$, i.e.,
$\I{S}{X} = \{U \in S \mid U \le X \,\lor\, X \le U\}$. 

\begin{lemma}\label{lem:deg_uppter_bound}
Let $C$ be a $(v,N,d;k)_q$ cdc and $X \le \Fqn$. Then we have 
$$
  \#\I{C}{X}\le  \left\{\begin{array}{rcl}
    A_q(\dim(X),d;k) &:& \dim(X)\ge k,\\
    A_q(v-\dim(X),d;k-\dim(X)) &:& \dim(X) < k.
  \end{array}\right.
$$
\end{lemma}
\begin{proof}
For the second part we write $V=X\oplus V'$ and $U_i=X\oplus U_i'$ for all $U_i\in C$. With this we have 
$\sdist(U_i,U_j)=2k-2\dim(U_i\cap U_j)\le 2\left(k-\dim(X)\right)-2\dim(U_i'\cap U_j')=\sdist(U_i',U_j')$.
\end{proof}

If $\#\I{C}{X}$ is small, then we can state the following upper bound for $\#C$: 

\begin{lemma}\label{lemma:deglebimpliescodesizelebound}
Let $(v,N,d;k)_q$ be a cdc $C$ and $0\le l\le v$. If $\#\I{C}{X} \le b$ for all $X \le \Fqn$ with $\dim(X)=l$, then 
$N \le \frac{\gaussmnum{v}{l}{q} b}{\gaussmnum{k}{l}{q}}$ if $l \le k$ and $N \le \frac{\gaussmnum{v}{l}{q} b}{\gaussmnum{v-k}{l-k}{q}}$ 
if $k \le l$.
\end{lemma}
\begin{proof}
Double counting the incidences between codewords $U\in C$ and subspaces $X$ with $\dim(X)=l$ gives
$
\gaussmnum{k}{l}{q} \cdot N
= \sum_X \#\I{C}{X} \le \sum_X b
= \gaussmnum{v}{l}{q} b
$
if $l\le k$ and 
$
\gaussmnum{v-k}{l-k}{q}\cdot N
= \sum_X \#\I{C}{X} \le
\sum_X b = \gaussmnum{v}{l}{q} b
$
if $l\ge k$.
\end{proof}

Now, we specialize our considerations to constant dimension codes with $v=2k$ and minimum subspace distance $d=2k-2$.
\begin{corollary}
\label{cor_one_incidence}
Let $C$ be a $(2k,N,2k-2;k)_q$ cdc for $k\ge 1$ and $c\in\mathbb{N}$. 
If $\#\I{C}{H} \le q^k+1 -c$ for all hyperplanes $H$ or $\#\I{C}{P} \le q^k+1 -c$ for all points $P$, then 
$N \le (q^k+1)(q^k+1 -c)$.
\end{corollary}
\begin{proof}
Apply Lemma~\ref{lemma:deglebimpliescodesizelebound} with $b=q^k+1-c$ and $l\in\{1,v-1\}$.
\end{proof}

Corollary~\ref{cor_one_incidence} will be applied in Section~\ref{main_sec} in order to deduce $A_2(8,6;4)\le 272$. 
In some cases it is computationally beneficial to consider the intersection of a subspace code with a hyperplane, see Section~\ref{sec_alternative_approach}.
\begin{lemma}(\cite[Lemma~2.8.i]{MR3543542})
\label{lemma_hyperplane}
Let $C$ be a $(v,N,d;K)_q$ subspace code and $P, H \le \Fqn$ with $\dim(P)=1$, $\dim(H)=v-1$, $P \not \le H$, and $d \ge 2$. Then the so-called \emph{shortened code}
$S(C,P,H) = \{U \cap H \mid U \in \I{C}{P} \} \cup \I{C}{H} $
is a $(v-1,\#\I{C}{P}+\#\I{C}{H},d';K')_q$ subspace code with $d' \ge d-1$ and $K' = (K \cup \{k-1 \mid k \in K\}) \cap \{0, 1, \ldots, v\}$.
\end{lemma}

Applying Lemma~\ref{lemma_hyperplane} for a $(v,N,d;k)_q$ cdc $C$ gives a $(v-1,N',d';\{k-1,k\})_q$ subspace code, where $d'\ge d-1$ and 
$N'=\#\I{C}{P}+\#\I{C}{H}$. For a more refined analysis we will consider incidences of codewords with pairs of points and hyperplanes.

\begin{proposition}
\label{prop_intermediate}
Let $S \subseteq \Gqnk$, $1 \le k \le v-1$, and $b \in \mathbb{N}$. If
$\#S > 
\frac{(q^v-1)(b-1)}{q^{v-k}+q^k-2}$, 
then there is a hyperplane $\bar{H}$ and a point $\bar{P}\not\leq \bar{H}$ 
with $\#\I{S}{\bar{H}} + \#\I{S}{\bar{P}} \ge b$.
\end{proposition}
\begin{proof}
Assume the contrary, i.e., $\#\I{S}{H} + \#\I{S}{P} \le  b-1$ for all pairs of points and hyperplanes $(P,H)$ 
with $P\not\leq H$. Double counting the triples $(P,H,U)$, where $U\in \I{\Gqnk}{H} \cup \I{\Gqnk}{P}$ gives 
\begin{align*}
\left(
\gaussmnum{v-k}{v-1-k}{q} \left(\gaussmnum{v}{1}{q}-\gaussmnum{v-1}{1}{q}\right)
+
\gaussmnum{k}{1}{q} \left(\gaussmnum{v}{v-1}{q}-\gaussmnum{v-1}{v-1-1}{q}\right)
\right)\cdot \# S
\\
=
\sum_P \sum_{H \in \gaussmset{\Fqn}{v-1} \setminus \I{\gaussmset{\Fqn}{v-1}}{P}} 
\Big(\#\I{S}{H} +\#\I{S}{P}\Big) \enspace ,
\end{align*}
noting that $\I{\Gqnk}{H} \cap \I{\Gqnk}{P}=\emptyset$, due to $P\not\leq H$. Using $\gaussmnum{a}{b}{q}=\gaussmnum{a}{a-b}{q}$ and 
$\#\I{S}{H} + \#\I{S}{P} \le  b-1$ we obtain
$$
\left(\gaussmnum{v-k}{1}{q}+\gaussmnum{k}{1}{q}\right)\cdot \left(\gaussmnum{v}{v-1}{q}-\gaussmnum{v-1}{v-1-1}{q}\right)
\cdot \# S
\le 
\gaussmnum{v}{1}{q} \left(\gaussmnum{v}{v-1}{q}-\gaussmnum{v-1}{v-1-1}{q}\right) \cdot (b-1) \enspace ,
$$
so that
$
  \# S\le \frac{
\gaussmnum{v}{1}{q}
(b-1)
}{
\gaussmnum{v-k}{1}{q} + \gaussmnum{k}{1}{q}
} 
=\frac{(q^v-1)(b-1)}{q^{v-k}+q^k-2}
$, 
which is a contradiction.
\end{proof}

Again, we specialize our considerations to constant dimension codes with $v=2k$ and minimum distance $d=2k-2$. Using the two well known facts 
$A_q(v,2k;k) = \frac{q^v-q}{q^k-1}-q+1$ for $v \equiv 1 \pmod{k}$ and $2 \le k \le v$ , due to a result on partial spreads, 
see \cite{MR0404010}, and $A_q(v,d;k)=A_q(v,d;v-k)$, due to the properties of orthogonal codes, we conclude:

\begin{corollary}\label{cor:ex_P_H}
For a $(2k,N,2k-2;k)_q$ cdc 
$C$ in $\Fqn$, where $k\ge 3$, we have $\#\I{C}{P} \le q^k+1$ and $\#\I{C}{H} \le q^k+1$ 
for all points $P$ and hyperplanes $H$. 
If $N > (q^k+1)(q^k+1 -(c+1)/2)$ for some 
$c\in\mathbb{N}$,   
then there is a hyperplane $\bar{H}$ and a point $\bar{P}$ 
with  
$\#\I{C}{\bar{H}} + \#\I{C}{\bar{P}} \ge 2(q^k+1) -c$ 
and $\bar{P}\not\leq\bar{H}$.
\end{corollary}
\begin{proof}
Lemma~\ref{lem:deg_uppter_bound} gives $\#\I{C}{P} \le A_q(2k-1,2k-2;k-1)=q^k+1$ 
and 
$\#\I{C}{H} \le A_q(2k-1,2k-2;k)=A_q(2k-1,2k-2;k-1)=q^k+1$. 
The seconds statement follows 
from Proposition~\ref{prop_intermediate} using $b=2(q^k+1) -c$.
\end{proof}
Corollary~\ref{cor:ex_P_H} will be applied in Section~\ref{sec_alternative_approach} in order to deduce $A_2(8,6;4)\le 272$.

\section{An integer linear programming bound for $A_2(8,6;4)$}
\label{main_sec}
Applying Corollary~\ref{cor_one_incidence} with $k=4$ gives the following facts. If all points or all hyperplanes are incident to at most $17-c$ 
codewords of an $(8,N,6;4)_2$ cdc $C$, then $N \le 17(17-c)$. In other words, if $N \ge 273$, then there is a point $\bar{P}$ and a 
hyperplane $\bar{H}$ that are incident to $17$ codewords in $C$, respectively. The $17$ codewords incident to $\bar{H}$ form a 
$(7,17,6;4)_2$ constant dimension code whose orthogonal is a $(7,17,6;3)_2$ cdc. The latter substructures have been completely classified up to 
isomorphism. 

\begin{theorem}(\cite[Theorem 5]{honold2016classification})
\label{thm_class_1}
$A_2(7,6;3)=17$ and there are 715 isomorphism types of $(7,17,6;3)_2$ constant dimension codes. Their automorphism groups have orders: 
$1^{551}2^{70}3^{27}4^{19}6^{6}7^{1}8^{8}12^{2}16^{7}24^{6}32^{5}42^{1}48^{5}64^{2}96^{1}112^{1}128^{1}192^{1}2688^{1}$.
\end{theorem}

These and all other classified constant dimension codes mentioned later on can be downloaded from the webpage of \cite{HKKW2016Tables}. The corresponding 
automorphism groups have been computed for this paper with the tool described in \cite{ubt_epub42}.

In general the determination of $A_q(v,d;k)$ can be formulated as an integer linear program (ILP), see e.g.~\cite{MR2796712}.
\begin{lemma}
\label{lemma_ILP}
  Let $q$ be a prime power, $v$, $0 \le k \le v/2$ non-negative integers and $d \le 2k$ a non-negative even integer. Using the abbreviations 
  $G:=\Gqnk$ and $\delta:=d/2$ the value of $A_q(v,d;k)$ coincides with the optimal target value of the binary linear program
\begin{align*}
\max
\sum_{U \in G} &x_U
\\
\st
\sum_{U \in \I{G}{A}} &x_U \le A_q(v-a,d;k-a)	&\forall A \in \gaussmset{\Fqn}{a} \forall a \in \{1, \ldots, k-\delta\}
\\
\sum_{U \in \I{G}{A}} &x_U \le 1				&\forall A \in \gaussmset{\Fqn}{a} \forall a \in \{k-\delta+1, k+\delta-1\}
\\
\sum_{U \in \I{G}{A}} &x_U \le A_q(v-a,d;k)		&\forall A \in \gaussmset{\Fqn}{a} \forall a \in \{k+\delta, \ldots, v-1\}
\\
&x_U \in \{0, 1\} 								&\forall U \in G
\end{align*}
\end{lemma}
The constraints are due to Lemma~\ref{lem:deg_uppter_bound} and correspond to clique constraints in an independent set formulation. We remark that 
the constraints corresponding to dimensions $a$ between $k-\delta+2$ and $k+\delta-2$ are redundant, i.e., they are implied by 
those for $a=k-\delta+1$ and $a=k+\delta-1$. The entire ILP consists of $\gaussmnum{v}{k}{q}$ binary variables, and
$\sum_{a=1}^{k-\delta+1} \gaussmnum{v}{a}{q} + \sum_{a=k+\delta-1}^{v-1} \gaussmnum{v}{a}{q}$ constraints. 

The linear programming (LP) relaxation of a binary linear program (BLP) $\max\{c^T x \mid A \cdot x \le b \land x \in \{0,1\} \}$ is 
given by $\max\{c^T x \mid A \cdot x \le b \land 0 \le x \le 1\}$. Note that the optimal value of an LP relaxation of an BLP is an upper 
bound for the objective function of the BLP.

Now we combine both approaches, i.e., we utilize the BLP from Lemma~\ref{lemma_ILP} and additionally prescribe each of the 
$715$ isomorphism types of $(7,17,6;4)_2$, i.e., $17$ variables $x_U$ are set to $1$, in separate computations. To this end, 
we remark that the hyperplane $\bar{H}$ can be chosen arbitrarily, since the group $\GL\left(\mathbb{F}_2^8\right)$ operates 
transitively on the set of hyperplanes.
The computation took 1021 hours on the cluster at the University of Bayreuth in parallel with at most 4 kernels. All objective values of the 
corresponding LP relaxations are between $206.2$ and $282.97$. The mean is approximately  $235.1$ with a standard deviation of roughly $5$. 
Only $6$ values are at least $255$: $255.67, 257.0, 258.75, 261.12, 268.04, 282.96$. Hence we conclude that $A_2(8,6;4) \le 282$.

In order to assume a hyperplane containing $17$ codewords, we have imposed $\#C\ge 273$, so that the tightest possible bound along those 
lines would be $A_2(8,6;4) \le 272$. To that end only the unique isomorphism type of a partial plane spread with LP objective value $282.96$ 
needs to be excluded. In principle one may just try to solve the corresponding BLP for this single case.

However, the following combinatorial relaxation turns out to be more promising. Consider $C' = \{U \cap \bar{H} \mid U \in C\}$, where 
the $17$ codewords, that are completely contained in $\bar{H}$, correspond to one of the previously not excluded isomorphism types of partial plane spreads. 
Being a bit more ambitious, we consider all four isomorphism types with an LP objective value of at least $258$.  
The prescribed $17$ codewords yield $17$ subspaces of dimension $4$ in $C'$ and all other codewords have dimension $3$. The pairwise intersection 
of $3$-dimensional codewords among themselves and with the $17$ $4$-dimensional subspaces is at most $1$-dimensional, due to $\sdist=6$. 
Since $\gaussmnum{7}{3}{2}=11811<200787=\gaussmnum{8}{4}{2}$ we get a much smaller problem. Moreover, the $17$ $4$-dimensional subspaces 
forbid many of the potential $3$-dimensional subspaces. Let $F^\perp$ be the orthogonal code of one of the 4 isomorphism types of $(7,17,6;3)_2$ 
codes which have an LP relaxation of at least $258$ and
$A(F):=\left.\left\{ U \in \gaussmset{\F_2^7}{3} \,\right|\, \dim(U \cap W) \le 1 \,\forall W \in F \right\}$. From the above we conclude
\begin{lemma}
  \label{lemma_ILP_hyperplane}
  If $C$ is an $(8,N,6;4)_2$ cdc containing the code $F$ in a hyperplane, then $N\le z(F)+\#F$, where
  \begin{align*}
z(F) = \max \sum_{U \in A(F)} &x_U \\
\sum_{U \in \I{A(F)}{L}} &x_U \le 1			&\forall L \in \gaussmset{\F_2^7}{2} \\
&x_U \in \{0,1\}							&\forall U \in A(F)
\end{align*}
\end{lemma}
The general benefit from a BLP formulation as in Lemma~\ref{lemma_ILP_hyperplane} is that the computation of $z(F)$ can be interrupted 
at any time still yielding an upper bound of $z(F)$. Spending 8~hours computation time on the BLP of Lemma~\ref{lemma_ILP_hyperplane} for 
each of the remaining $4$ subproblems yields the following results:

\begin{center}
\begin{tabular}{llll}
$\#\Aut$	&		LP bound Lemma~\ref{lemma_ILP} & $\#A(F)$ 	& $z(F)+17\le$ 	\\
\hline
24			&		258.75 					& 900 		& 250.31 											\\ 
4			&		261.12 					& 896 		& 255.43 											\\ 
32			&		268.04 					& 948 		& 259.67 											\\ 
64			&		282.96 					& 864 		& 267.67 											\\ 
\end{tabular}
\end{center}

Hence we conclude that $A_2(8,6;4) \le 272$. We remark that the stated computation times heavily depend on the used 
(I)LP solver 
and that the case $F=\emptyset$ in Lemma~\ref{lemma_ILP_hyperplane} corresponds to the determination of 
$A_2(7,4;3)$, where $333\le A_2(7,4;3)\le 381$ is known \cite{HKKW2016Tables}.

\section{Alternative ways to 
prove $A_2(8,6;4) \le 272$}
\label{sec_alternative_approach}
In this section we want to present alternative approaches to computationally prove $A_2(8,6;4) \le 272$. Given the needed 
1053~hours of computation time of the approach of Section~\ref{main_sec}, an independent verification might not be a bad idea. Especially, numerical 
algorithms based on floating point numbers might be considered to be suspicious. So, we try to minimize the number of those computations. However, our 
main motivation is to present different algorithmic approaches that might be beneficial for other parameters.

While the approach of Section~\ref{main_sec} is based on the classification of $(7,17,6;3)_2$ constant dimension codes, using 
Lemma~\ref{lemma_hyperplane},  we can also start with a classification of the $(7,34,5;\{3,4\})_2$ subspace codes.
 
\begin{theorem}(\cite[Theorem~3.3.ii]{MR3543542}, \cite[Theorem~6]{honold2016classification})
\label{thm_class_2}
$A_2(7,5;\{0,\dots,7\})=34$ and there are exactly $20$ isomorphism types of codes having these parameters. 
All of them have dimension distribution $3^{17} 4^{17}$. In $11$ cases, the automorphism group is trivial and 
in the remaining $9$ cases, the automorphism group is a unique group of order $7$, which partitions $\mathbb{F}_2^7$ 
into $2$ fix points and $18$ orbits of size $7$.
\end{theorem}

Applying Corollary~\ref{cor:ex_P_H} with $q=2$, $k=4$, and $c=0$ gives that every $(8,N,6;4)_2$ code with $N>280.5$ has 
to contain a hyperplane whose intersection with the code is a $(7,34,5;\{3,4\})_2$ subspace code. The corresponding $20$ 
isomorphism types contain just nine of of the $715$ isomorphism types of $(7,17,6;3)_2$ and $(7,17,6;4)_2$ constant dimension 
codes. Denoting these nine cases by $a_1,\dots,a_9$, the $20$ isomorphism types of $(7,34,5;\{3,4\})_2$ subspace codes 
can be categorized as 
$$\{\{a_1,a_6\},\{a_2,a_6\},\{a_3,a_7\},\{a_3,a_8\},\{a_4,a_4\},\{a_4,a_9\},\{a_5,a_6\},\{a_6,a_6\}\}.
$$ 
In particular, these pairings can be covered by just the three cases $\{a_3,a_4,a_6\}$. Prescribing the corresponding $17$ 
codewords and computing the LP relaxation of Lemma~\ref{lemma_ILP} gives:\\[-7mm] 

\begin{center}
\begin{tabular}{lll}
type & $\#\Aut$	&		LP bound Lemma~\ref{lemma_ILP}	\\
\hline
$a_4$	& 32		&		221.00\\
$a_6$	& 7		&		230.63\\
$a_3$	& 32		&		268.04\\
\end{tabular}
\end{center} 

Thus, by computing three linear programs, we can conclude $A_2(8,6;4) \le 280$. We remark that the classification results of 
Theorem~\ref{thm_class_1} and Theorem~\ref{thm_class_2} were obtained using the clique search software \texttt{cliquer 1.21} \cite{niskanen2003cliquer}, which is not 
based on floating point numbers.  

An upper bound for $A_2(8,6;4)$ based on Corollary~\ref{cor:ex_P_H} with $q=2$, $k=4$, and $c=1$ needs the classification of all 
$(7,33,5;\{3,4\})_2$ subspace codes.

\begin{theorem}
\label{thm_class_3}
There are $563$ isomorphism types of $(7,33,5;\{3,4\})_2$ 
codes. Their automorphism groups have
orders: $1^{481}2^{19}4^{4}7^{56}8^{1}14^{2}$. The possible dimension distributions are $3^{16} 4^{17}$ 
and $3^{17} 4^{16}$, both appearing for a code and its orthogonal.
\end{theorem}
\begin{proof}
  For each of the $715$ isomorphism types of $(7,17,6;3)_2$ constant dimension codes $C$ in $\mathbb{F}_2^7$ 
  we first compute $A(C)=\left.\left\{W \in \gaussmset{\mathbb{F}_2^7}{4} \,\right|\, d_S(W,U) \ge 5 \forall U \in C \right\}$. 
  Then, we build up a graph $\mathcal{G}(C)$ with vertex set $A(C)$. Two different vertices $U,W\in A(C)$ are joined by 
  an edge iff $\sdist(U,W) \ge 6$. These $715$ graphs have between $832$ and $1056$ vertices and between $213760$ and $353088$ 
  edges. Applying the software package \texttt{cliquer 1.21} \cite{niskanen2003cliquer} on the computing cluster of the University 
  of Bayreuth 
  gives $23740$ cliques of cardinality $16$ each -- after 11,200 hours of computational time. Via the group action of the automorphism 
  group of the corresponding $(7,17,6;3)_2$ cdc $C$, they form $563$ orbits. 
\end{proof}

We remark that $76$ out of the $715$ isomorphism types of $(7,17,6;3)_2$ codes can be extended to $(7,33,5;\{3,4\})_2$ codes 
having automorphism groups of orders $1^{51}2^{7}3^{3}4^{2}6^{1}7^{1}12^{1}16^{2}32^{2}42^{1}64^{1}112^{1}128^{1}192^{1}2688^{1}$ and 
extensions of frequencies $1^{56}2^{7}3^{1}4^{1}5^{2}6^{1}10^{1}11^{1}44^{1}49^{1}67^{1}77^{1}104^{1}108^{1}$. In $75$ of these $76$ 
cases the LP relaxation of Lemma~\ref{lemma_ILP} gives an objective value strictly smaller than $272$, so that only one case with 
LP relaxation $282.96$ and $\#\Aut=64$ remains. As described in Section~\ref{main_sec}, we can apply the BLP of 
Lemma~\ref{lemma_ILP_hyperplane}. Thus, besides exact arithmetic clique computations, $75$ LP computations and a single BLP computation 
suffices to deduce $A_2(8,6;4) \le 272$.  

Instead of decomposing the $563$ isomorphism types of $(7,33,5;\{3,4\})_2$ codes into their components, we may also utilize 
the following BLP formulation.
\begin{lemma}
  \label{lemma_ILP_hyperplane_blowup}
  If $C$ is an $(8,N,6;4)_2$ cdc containing the $(7,17,6;3)_2$ code $F_3$ and $(7,16,6;4)_2$ code $F_4$ in the hyperplane $\im(\iota)$ then $N\le z(F_3,F_4)$, where $\iota: \F_2^7 \rightarrow \F_2^8, v \mapsto (v \mid 0)$, $G:=\gaussmset{\Fqn}{4}$, $Q:=\gaussmset{\Fqn}{1} \setminus \I{\gaussmset{\Fqn}{1}}{\im(\iota)}$, and
  \begin{align*}
z(F_3,F_4) = &\max \sum_{U \in G} x_U \quad\text{st}&
  x_U = 1			&\quad\forall U\in \iota(F_4) \\
  \!\!\!\!\sum_{U' \in \iota(F_3)}\!\!x_{\langle U',P\rangle} = y_P			&\quad\forall P \in Q &
  \sum_{P\in Q}y_{P} = 1			&  \\
  \!\!\!\!\sum_{U \in \I{G}{A}}\!\!\!\!\!\!\!\! x_U \le 17	&\quad\forall A \in \gaussmset{\Fqn}{a} \forall a\in\{1,7\} &
  \!\!\!\!\sum_{U \in \I{G}{A}}\!\!\!\!\!\!\!\! x_U \le 1	&\quad\forall A \in \gaussmset{\Fqn}{a} \forall a\in\{2,6\} \\
x_U \in \{0,1\}							&\quad\forall U \in G &y_P \in \{0,1\}							&\quad\forall P \in Q\\
\end{align*}
\end{lemma}
Of course, we also obtain $z(F_3,F_4)\le 272$ in all $563$ cases.

\section{Conclusion}
\label{sec_conclusion}
In this paper we have applied integer linear programming techniques in order to improve the upper bound of $A_2(8,6;4)$ from $289$ to $272$. 
While ILP solvers generally struggle with formulations involving a large automorphism group, we have utilized the huge symmetry group 
of the underlying metric space in order to exhaustively enumerate certain substructures up to isomorphism in a first step. In the second
step, prescribing such a substructure removes much of the initial symmetry of the ILP formulation, so that ILP solvers might successfully be 
applied. Here the general key question is to find appropriate substructures. Of course one may go by existing classification results. 
In Theorem~\ref{thm_class_3} we have obtained another such classification result.  
Additionally, we have considered a combinatorial relaxation in Lemma~\ref{lemma_ILP_hyperplane}, which turned out be rather strong. 
Since the current gap $257\le A_2(8,6;4)\le 272$ is still large, the presented algorithmic approaches should be further developed. 
To this end, we remark that the $(7,16,6;3)_2$ codes have also been classified in \cite{honold2016classification} and refer to 
the forthcoming paper \cite{forthcoming_all}, where also implications for the classification of MRD codes and other parameters of 
subspace codes are considered.

Given the bounds $A_2(6,4;3)\le 77$ and $A_2(8,6;4)\le 272$, one might conjecture that $A_{2}(2k,2k-2;k)$ is much smaller 
than $\left(2^k+1\right)^2$, which is implied by the Johnson bound and Beutelspacher's result for partial spreads, for 
increasing $k\ge 3$. Unfortunately, those results yield no improvements for other upper bounds for constant dimension codes 
based on the Johnson bound.

\begin{lemma}
No improvement on the upper bound of $A_q(2k,2k-2;k)$ for $k \ge 3$ yields a stronger bound on $A_q(2k+1,2k-2;k)$ 
as $A_q(2k+1,2k-2;k) = A_q(2k+1,2k-2;k+1) \le \left\lfloor\frac{q^{2k+1}-1}{q^{k+1}-1} A_q(2k,2k-2;k)\right\rfloor$, which is 
implied by the Johnson bound.
\end{lemma}
\begin{proof}
  Due to the Johnson bound and $A_q(2k,2k-2;k-1) \le \frac{q^{2k}-1}{q^{k-1}-1}$, we have 
  \begin{align*}
A_q(2k+1,2k-2;k)
\le
\left\lfloor\frac{q^{2k+1}-1}{q^k-1} A_q(2k,2k-2;k-1)\right\rfloor
\le
\frac{q^{2k+1}-1}{q^k-1}
\cdot
\frac{q^{2k}-1}{q^{k-1}-1}
\\
<
\frac{q^{2k+1}-1}{q^{k+1}-1}
\cdot
q^{2k}
\le
\left\lfloor \frac{q^{2k+1}-1}{q^{k+1}-1} \cdot (q^{2k} +1) \right \rfloor
\le
\left\lfloor\frac{q^{2k+1}-1}{q^{k+1}-1} A_q(2k,2k-2;k)\right\rfloor,
\end{align*}
where we have used $A_q(2k,2k-2;k)\ge q^{2k} +1$, which is obtained from a lifted MRD code extended by an additional codeword. 
\end{proof}  
With respect to possible improvements on $1025 \le A_2(10,8;5) \le 1089$, we remark that, up to our knowledge, the 
$(9,33,8;4)_2$ constant dimension codes have not been classified and the gap $65\le A_2(9,5;\{4,5\})\le 66$ has not been closed yet.  

\section*{Acknowledgment}
The authors would like to thank Michael Kiermaier for preprocessing the codes of the used classification results and supporting us 
in using the software package to compute automorphism groups of constant dimension codes from \cite{ubt_epub42}.  
Further thanks go to the \emph{IT service center} of the University Bayreuth 
for providing the excellent computing cluster and especially Dr. Bernhard Winkler for his support.


\end{document}